\documentclass[leqno]{article}

\usepackage{amsfonts,amssymb,amsmath,amsthm}
\usepackage[english]{babel}
\usepackage{algorithm}
\usepackage{algorithmic}
\usepackage{nicefrac}
\usepackage{authblk}
\usepackage[small]{titlesec}
\usepackage{subcaption}
\usepackage[textsize=tiny]{todonotes}

\newtheorem{theorem}{Theorem}[section]
\newtheorem{definition}[theorem]{Definition}
\newtheorem{lemma}[theorem]{Lemma}
\newtheorem{corollary}[theorem]{Corollary}

\newtheoremstyle{mystyle}
{}
{}
{\normalfont}
{}
{\normalfont\itshape}
{.}
{ }
{}

\theoremstyle{mystyle}
\newtheorem{example}[theorem]{Example}

\numberwithin{equation}{section}

\newcommand{\R}{\mathbb{R}}
\newcommand{\Q}{\mathbb{Q}}
\newcommand{\Z}{\mathbb{Z}}
\newcommand{\N}{\mathbb{N}}

\newcommand{\DD}{U(\sigma)}
\newcommand{\DM}{U(L,l)}

\renewcommand{\d}{\textnormal{d}}
\newcommand{\indodd}[1]{\,#1\;\textnormal{odd}\,}
\newcommand{\indeven}[1]{#1\;\textnormal{even}}

\newcommand{\conv}{\operatorname{conv}}
\newcommand{\cconv}{\overline\conv}

\newcommand{\aei}{a.e.\ in}
\newcommand{\testfunc}[1]{{\cal C}^\infty_c(#1)}
\newcommand{\dext}{\textnormal{ext}}
\newcommand{\dlin}{\textnormal{lin}}
\newcommand{\avn}[2]{{#1}_{[#2]}}
\renewcommand{\mid}{\colon}

\providecommand{\keywords}[1]{\textbf{Keywords.} #1}

\title{\bf Extended Formulations for\\ Binary Optimal Control Problems
  \thanks{This work was partially supported by Deutsche
    Forschungsgemeinschaft under grant no.~BU~2313/7-1. The author
    would like to thank Alexandra Gr\"utering and Christian Meyer for
    many valuable comments that helped to improve this paper
    significantly.}}

\author{Christoph Buchheim}
\affil{Fakultät für Mathematik, Technische Universität Dortmund\\
{\tt christoph.buchheim@math.tu-dortmund.de}}
\date{}

\begin{document}
\maketitle

\begin{abstract}
  Extended formulations are an important tool in polyhedral
  combinatorics. Many combinatorial optimization problems require an
  exponential number of inequalities when modeled as a linear program
  in the natural space of variables. However, by adding artificial
  variables, one can often find a small linear formulation, i.e.,
  one containing a polynomial number of variables and constraints,
  such that the projection to the original space of variables yields
  a perfect linear formulation. Motivated by binary optimal control
  problems with switching constraints, we show that a similar approach
  can be useful also for optimization problems in function space, in
  order to model the closed convex hull of feasible controls in a
  compact way. More specifically, we present small extended
  formulations for switches with bounded variation and for dwell-time
  constraints. For general linear switching point constraints, we
  devise an extended model linearizing the problem, but show that a
  small extended formulation that is compatible with discretization
  cannot exist unless~P$\,=\,$NP.

  \smallskip
  
\end{abstract}

\section{Introduction}\label{sec:intro}

Optimal control problems with discrete-valued control variables are a
rather recent topic in infinite-dimensional optimization. The
joint consideration of ODE- or PDE-constraints with combinatorial
restrictions leads to new challenges and insights both on the optimal
control and on the discrete optimization side. In this paper, we focus
on the setting where a binary control function varies over a
continuous time horizon~$[0,T]$ and assume that a set of admissible
controls~$U$ is given, containing all feasible switching
patterns~$u\colon [0,T]\to\{0,1\}$. Moreover, we assume throughout
that all~$u\in U$ have bounded variation, i.e., that~$u$ switches
between~$0$ and~$1$ only a finite number of times. However, unlike in
other approaches, we do not discretize the problem a priori, so that
every point in~$[0,T]$ remains a potential switching point. A rigorous
formulation of these assumptions requires a formal introduction of the
function spaces in which we model the set~$U$, for this we refer to
Section~\ref{sec:prelim}.

While many approaches presented in the literature for solving optimal
control problems over such a set~$U$ are heuristic and depend on a
predefined discretization~\cite{KMS11,ZS20,ZRS20}, we recently
proposed a branch-and-bound algorithm for solving such problems to
global optimality in function space~\cite{partI,partII,partIII}. The
core of the latter approach is the understanding of the closed convex
hull~$\cconv(U)$ of the feasible set~$U$ and its outer description by
linear inequalities in function space. Depending on the structure
of~$U$, the separation problem for~$\cconv(U)$ may be intractable even
for a given discretization. For two classes of constraints, namely for
the case of bounded variation and for the so-called minimum dwell-time
constraints, we however exploit that the separation problem after
discretization can be solved in polynomial time. Both types of
constraints have been considered in the literature before. The former
constraint just bounds the number of switchings by a
constant~\cite{KMS11,ZS20}, while the latter requires that two
switchings do not follow each other too closely in time~\cite{ZRS20}.

From a combinatorial optimization perspective, this situation suggests
a closer look at the complexity of a given class of switching
constraints. It turns out, however, that a simple extension of the
NP-hardness concept of finite-dimensional optimization to function
spaces is not possible, since the complexity of the discretized
problem may actually depend on the discretization, or, more precisely,
on the number of grid cells used; see Example~\ref{ex:vc}. Nevertheless,
the tractability of relevant classes of switching constraints in
finite dimension paves the way for the transfer of well-studied tools
from discrete optimization to the infinite-dimensional setting.

One such tool are extended formulations. Many combinatorial
optimization problems require an exponential number of inequalities
when modeled as a linear program in the natural space of
variables. However, by adding artificial variables, one can often find
a small linear formulation, i.e., one containing a polynomial number
of variables and constraints, such that the projection to the original
space of variables yields the convex hull of the original feasible
set; see~\cite{CCZ13} for a survey containing both examples and abstract
results. In the present paper, we argue that the same approach can be useful
also for optimization problems in function space. More specifically,
our aim is to devise small extended formulations in function space
for some relevant types of constraints~$U$, such that the projection
to the original space of variables agrees with~$\cconv(U)$; see
Section~\ref{sec:extform} for a precise definition.

For the above-mentioned cases of bounded variation and dwell-time
constraints, we show that such small extended formulations indeed
exist; see Section~\ref{sec:bounded} and Section~\ref{sec:dwell}. Both
proofs are based on corresponding results on the existence of extended
formulations in the finite-dimensional case. Roughly speaking, the
extended formulations in function space can be viewed as limits of the
finite-dimensional formulations when the number of grid cells goes to
infinity. In particular, the formulations are compatible with
discretization in a certain sense. The use of these extended
formulations within the branch-and-bound algorithm of~\cite{partIII}
could be an alternative for the outer approximation algorithm
presented in~\cite{partII}, which requires the repeated dynamic
separation of violated cutting planes and subsequent
re-optimizations. The extended formulations lead to exactly the same
dual bounds without the need of any separation loop.

It follows that, for the two mentioned types of constraints,
discretization leads to extended formulations in finite dimension that
are already known in the literature. However, the advantage of having
at hand an extended formulation in function space, as opposed to the
first-discretize-then-optimize paradigm, is twofold. Firstly, it shows
that ``in the limit'' the model is consistent, so that artifacts
arising only from the discretization are avoided. E.g., it may happen
that all discretizations of an optimal control problem admit optimal
solutions while the original problem in function space does not. This
shows the necessity of having a well-defined model in function space
in the first place before considering its discretizations. Secondly,
discretization leads to a smaller feasible region and thus only allows
to derive a primal bound for the problem in function space. Instead,
an extended formulation in function space could, e.g., be dualized in
function space and then discretized, so that a safe~\emph{dual} bound
could be computed, which would be much more valuable within a
branch-and-bound algorithm or similar approaches.

The situation is different for general switching point constraints as
considered in Section~\ref{sec:general}. Here, we parametrize the
binary control by the finitely many switching points and require that
these switching points satisfy certain linear constraints; this
generalizes the minimum dwell-time constraints. For this class, we show that a
small extended formulation that is compatible with discretization
cannot exist unless~P$\,=\,$NP. This is not only the case for a fixed
discretization, but also when the discretization may be refined. This
result is essentially obtained by showing that the corresponding
sets~$U$ after discretization become intractable, i.e., it is NP-hard
to optimize a given linear objective function over the latter sets.

\section{Function Spaces and Discretization}\label{sec:prelim}

In order to show our main results on the existence of extended
formulations, or even to define the concept of extended formulations
in infinite dimension, we first need to introduce some notation
concerning appropriate function spaces and discretization. However, we
will limit ourselves to essential definitions and observations
here. For more details, we refer the reader to the
monographs~\cite{ALT16,Bre11} or~\cite{ATT14}.

\subsection{The Space $L^2$ and Distributional Derivatives}

Given an open subset~$\Omega\subseteq\R$, we address optimization
problems in the reflexive Banach space~$L^2(\Omega)$ consisting of all
Lebesgue measurable functions~$u\colon\Omega\to\R$ such that~$|u|^2$
is Lebesgue integrable, equipped with the norm
\[
||u||_{L^2(\Omega)}:=\Big(\int_\Omega |u(t)|^2\,\d t\Big)^{\nicefrac 12}\;.
\]
In~$L^2(\Omega)$, two functions are identified if they only differ on
a Lebesgue null set, i.e., elements of~$L^2(\Omega)$ are formally
defined as equivalence classes of functions. In particular, pointwise
evaluations are not well-defined for~$u\in L^2(\Omega)$, and all
constraints on~$u$ discussed in the following can only be required
almost everywhere (a.e.)  in~$\Omega$. Nevertheless, as is common, we
will frequently specify an element~$u\in L^2(\Omega)$ by a pointwise
definition of a representative of the equivalence class~$u$.

We write~$u_n\to u$ if~$u_n$ converges strongly to~$u$
in~$L^2(\Omega)$, i.e., if~$||u-u_n||_{L^2(\Omega)}\to 0$
for~$n\to\infty$. For~$u\in L^2(\Omega)$, the distributional
derivative~$Du$ is the linear functional on $\testfunc\Omega$
defined~by
\[
Du(\varphi):=-\int_\Omega u(t)\varphi'(t)\,\d t\quad\forall
\varphi\in\testfunc\Omega\;,
\]
where~$\testfunc\Omega$ denotes the set of all smooth
functions~$\varphi\colon\Omega\to\R$ with compact support.  If there
exists a function~$w\in L^2(\Omega)$ such that
\[
Du(\varphi)=\int_\Omega w(t)\varphi(t)\,\d t\quad\forall
\varphi\in\testfunc\Omega\;,
\]
it is called the weak derivative of~$u$ and denoted by~$u'$. If~$u$ is
a differentiable function in the classical sense, the weak
derivative~$u'$ agrees with the usual deriative, which follows from
integration by parts.

In the following, we will write~$Du\ge 0$ if
\[
Du(\varphi)\ge 0\quad\forall \varphi\in\testfunc\Omega,\varphi\ge 0\;.
\]
In this case, the function~$u$ is monotonically increasing (outside a
null set). The following observation shows that non-negativity of the
distributional derivative is a closed condition.
\begin{lemma}\label{lem:nndu}
  Let~$u_n\in L^2(\Omega)$ with~$Du_n\ge 0$ for all~$n\in\N$. If~$u_n$
  converges strongly to~$u$ in~$L^2(\Omega)$, then~$Du\ge 0$.
\end{lemma}
\begin{proof}
  Given any fixed~$\varphi\in\testfunc\Omega$, first note that
  \begin{eqnarray*}
    \Big|\int_\Omega u_n(t)(-\varphi'(t))\,\d t-\int_\Omega
    u(t)(-\varphi'(t))\,\d t\Big| & \le &
    ||u_n-u||_{L^2(\Omega)}\cdot||\varphi'||_{L^2(\Omega)}\to 0
  \end{eqnarray*}
  for~$n\to\infty$ since~$||u_n-u||_{L^2(\Omega)}\to 0$ by assumption.
  If~$\varphi\ge 0$, we then have
  \[
  Du(\varphi) = \int_\Omega u(t)(-\varphi'(t))\,\d t =
  \lim_{n\to\infty}\int_\Omega u_n(t)(-\varphi'(t))\,\d t =
  \lim_{n\to\infty}Du_n(\varphi)\ge 0\;,
  \]
  which shows~$Du\ge 0$.
\end{proof}
To conclude this subsection, we mention a technical statement that
will be helpful in some of the following proofs. Here, we use the
Sobolev space~$H^1_0(\Omega)$, which can be defined as the closure
of~$\testfunc{\Omega}$ with respect to the norm
$$||\varphi||_{H^1(\Omega)}:=\big(||\varphi||_{L^2(\Omega)}^2+||\varphi'||_{L^2(\Omega)}^2\big)^{\nicefrac
  12}\;,$$
where~$\varphi'$ denotes the weak derivative of~$\varphi$
defined above.  The latter definition, together with the definition
of~$Du\ge 0$ given above, immediately implies
\begin{lemma}\label{lem:approx_cont}
  Let~$u\in L^2(\Omega)$ with~$Du\ge 0$ and let $\varphi\in
  H^1_0(\Omega)$ such that~$\varphi\ge 0$. Then we have~$-\int_\Omega
  u(t)\varphi'(t)\,\d t\ge 0$.
\end{lemma}
For a detailed introduction to Sobolev spaces, we refer the reader to
the literature~\cite{ALT16,ATT14,Bre11}.

\subsection{Functions of Bounded Variation}\label{sec:defs}

Throughout this paper, we will consider functions of bounded variation
defined on~$\Omega$. For a precise definition, consider the seminorm
on~$L^2(\Omega)$ given by
\[
|u|_{BV(\Omega)}:=\sup_{\substack{\varphi\in\testfunc\Omega\\ ||\varphi||_\infty\le 1}}\, \int_\Omega u(t)\varphi'(t)\,\d t\;.
\]
We then define
\[
BV(\Omega):=\{u\in L^2(\Omega)\mid |u|_{BV(\Omega)}<\infty\}\;.
\]
The distributional derivative~$Du$ of a function~$u\in BV(\Omega)$ can
be represented by a finite signed regular Borel measure
on~$\Omega$. More formally, for each~$u\in BV(\Omega)$ there exists
such a measure~$\mu$ with
\[
Du(\varphi)=\int_\Omega \varphi\,\d \mu\quad\forall
\varphi\in\testfunc\Omega\;,
\]
which we will denote by~$\partial u$ in the following. The Jordan
decomposition theorem then allows to write~$\partial u=(\partial
u)^+-(\partial u)^-$, where~$(\partial u)^+$ and~$(\partial u)^-$ are
non-negative measures on~$\Omega$, and this decomposition is unique.
Setting~$|\partial u|:=(\partial u)^++(\partial u)^-$ and
using~\cite[Theorem~6.26]{ALT16} we obtain $|u|_{BV(\Omega)}=|\partial
u|(\Omega)$, i.e., the variation of~$u$ is the total variation of the
measure~$\partial u$.

In our proofs, we will make extensive use of~$u^+,u^-\in L^2(\Omega)$
defined by
\[
u^+(t):=(\partial
u)^+\big(\Omega\cap(-\infty,t]\big),~u^-(t):=(\partial
  u)^-\big(\Omega\cap(-\infty,t]\big)\;.
\]
Since~$u^+$ is Lebesgue measurable and $||u^+||_{L^\infty(\Omega)}\le
(\partial u)^+(\Omega)$, we can use Fubini's theorem and obtain
\begin{eqnarray*}
(Du^+)(\varphi) & = & -\int_\Omega u^+(t)\varphi'(t)\,\d t =
  -\int_\Omega\Big(\int_\Omega\chi_{(-\infty,t]}(\tau)\,\d(\partial
    u)^+(\tau)\Big)\varphi'(t)\,\d t\\ & = & -\int_\Omega
    \Big(\int_\Omega
    \underbrace{\chi_{(-\infty,t]}(\tau)}_{=\chi_{[\tau,\infty)}(t)}\varphi'(t)\,\d
        t\Big)\,\d(\partial u)^+(\tau) = \int_\Omega
        \varphi(\tau)\,\d(\partial u)^+(\tau)
\end{eqnarray*}
for all~$\varphi\in\testfunc{\Omega}$, so that~$Du^+$ is represented
by the measure~$(\partial u)^+$ and analogously~$Du^-$ is represented
by~$(\partial u)^-$.  Moreover, it follows that
$|u^+|_{BV(\Omega)}+|u^-|_{BV(\Omega)}=|u|_{BV(\Omega)}$ and
hence~$u^+,u^-\in BV(\Omega)$ with $\partial(u^+)=(\partial u)^+\ge
\partial u$ and $\partial(u^-)=(\partial u)^-\ge -\partial u$.

In this paper, we are mostly interested in functions~$u\in BV(\Omega)$
with $u(t)\in\{0,1\}$ for almost all~$t\in \Omega$. Up to null sets,
these functions can thus be viewed as binary switches that change only
a finite number of times in~$\Omega$. In our formulations, we will
assume moreover that the considered time horizon is~$[0,T]$,
for~$T\in\Q_+$, and that the switch is zero at the beginning of this
time horizon. For modeling this, given any~$S\subseteq\R$, we
introduce the notation
\[
BV_{\star}(0,T;S):=\{u\in BV(-\infty,T)\mid u=0\text{
  \aei\ }(-\infty,0),~u\in S\text{ \aei\ }(0,T)\}\;.
\]
Here and in the following, we shortly write that $u$ has some property
\aei\ $\Omega$ if $u(t)$ has this property for almost all
$t\in\Omega$.  Now $u\in BV_{\star}(0,T;S)$ ensures that~$\partial
u\big((-\infty,0)\big)=0$ and
\[
|u|_{BV(-\infty,T)}=|\partial u|(\{0\})+|u|_{BV(0,T)}
\]
for all~$u\in L^2(-\infty,T)$, i.e., the initial value of~$u$ in the
time horizon~$[0,T]$ is taken into account. Analogously, we will use
the notation
\[
L^2_{\star}(0,T;S):=\{u\in L^2(-\infty,T)\colon u=0\text{
  \aei\ }(-\infty,0),~u\in S\text{ \aei\ }(0,T)\}
\]
and shortly write~$L^2_{\star}(0,T)$ and~$BV_{\star}(0,T)$
for~$L^2_{\star}(0,T;\R)$ and~$BV_{\star}(0,T;\R)$, respectively. By
these definitions, each function~$u\in BV_{\star}(0,T)$ now has a
representative~$t\mapsto \partial u\big([0,t]\big)$.

\subsection{Discretization and Approximation by Piecewise Constant Functions}\label{sec:app}

Since we are mostly interested in binary functions, a natural approach
for discretization uses piecewise constant functions. For simplicity,
we will concentrate on equidistant grids throughout the paper. Given a
number~$N\in\N$ of grid cells and~$v\in\R^N$, let~$\overline v\in
L^2_{\star}(0,T)$ be the piecewise constant function defined by
\[
\overline v= v_i\text{ on }\big((i-1)\tfrac TN,i\tfrac
TN\big),~i=1,\dots,N\;.
\]
Conversely, given~$u\in L^2_{\star}(0,T)$, let~$u_N\in\R^N$ be defined
by the piecewise averages
\[
(u_N)_i:=\tfrac NT\int_{(i-1)\frac TN}^{i\frac TN} u(t)\,\d t
\]
for~$i=1,\dots,N$. Clearly, we then have~$(\overline v)_N=v$ for
all~$v\in\R^N$. For the following, we further introduce the
notation~$\avn uN:=\overline{{\scriptstyle (}u_N{\scriptstyle )}}\in
L^2_{\star}(0,T)$ for~$u\in L^2_{\star}(0,T)$, i.e.,~$\avn uN$ arises
from~$u$ by replacing its function values by their piecewise averages
on the intervals~$\big((i-1)\tfrac TN,i\tfrac TN\big)$.
\begin{definition}\label{def:disc}
  Let~$U\subseteq L^2_{\star}(0,T)$ and~$N\in\N$. Then the
  \emph{discretization} $U_N$ of~$U$ is defined as the subset $U\cap
  \{\bar u\mid u\in\R^N\}$ of~$U$.
\end{definition}
By this definition, the discretization~$U_N$ consists of piecewise
constant functions. However, we will sometimes identify~$U_N$ with a
subset of~$\R^N$, namely via the correspondence between a
vector~$v\in\R^N$ and the piecewise constant function~$\bar v$.
\begin{lemma}\label{lem:convdisc}
  For every~$U\subseteq L^2_{\star}(0,T)$, we
  have~$\conv(U_N)\subseteq \conv(U)_N$.
\end{lemma}
\begin{proof}
  Let~$u\in\conv(U_N)\subseteq L^2_{\star}(0,T)_N$ and choose
  $\lambda_i\ge 0$ and $u^{i}\in U_N\subseteq U$ for~$i=1,\dots,k$
  such that~$\sum_{i=1}^k\lambda_i=1$ and~$u=\sum_{i=1}^k
  \lambda_iu^{i}$. Then~$u\in\conv(U)$ and hence~$u\in\conv(U)_N$.
\end{proof}
Note that~$\conv(U_N)\neq\conv(U)_N$ in general. For a simple example,
let~$U=\{\chi_{[0,1]},\chi_{[1,2]}\}$ with~$T=2$, where~$\chi_A$
denotes the characteristic function
of~$A\subseteq\R$. Then~$U_N=\emptyset$ for all odd~$N\in\N$, since
each~$u\in U_N$ must be constant on~$[\tfrac{N-1}N,\tfrac {N+1}N]$,
but then~$u\not\in U$. On the other hand, we have~$\tfrac
12\chi_{[0,2]}\in \conv(U)_N$ for all~$N$.

The following observation is a central ingredient in our proofs
presented in the subsequent sections. It shows that a function~$u\in
L^2(0,T)$ can be approximated by its piecewise averages provided
that~$u$ has bounded variation; see Figure~\ref{fig:app} for an
illustration.
\begin{figure}
  \begin{center}
    \begin{tikzpicture}[scale=1.5]
      \draw[->,color=gray!75] (0,0) -- (0,1);
      \draw[->,color=gray!75] (0,0) -- (3,0);
      \draw[thick,domain=0:1,color=blue] plot (\x,{\x*\x});
      \draw[thick,domain=1:1.4,color=blue] plot (\x,{0.8});
      \draw[thick,domain=1.4:2,color=blue] plot (\x,{0.6});    
      \draw[thick,domain=2:3,color=blue] plot (\x,{0.5+\x/7*sin(2*3.14159*(\x-1.375) r))});
    \end{tikzpicture}
    \qquad
    \begin{tikzpicture}[scale=1.5]
      \draw[domain=0:1,color=blue!50] plot (\x,{\x*\x});
      \draw[domain=1:1.4,color=blue!50] plot (\x,{0.8});
      \draw[domain=1.4:2,color=blue!50] plot (\x,{0.6});    
      \draw[domain=2:3,color=blue!50] plot (\x,{0.5+\x/7*sin(2*3.14159*(\x-1.375) r))});
      \foreach \i / \j in {0/0.1875,1/0.7636,2/0.4756,3/0.6125} {
        \draw[color=gray,dotted,thick] (\i*3/4,0) -- (\i*3/4,1);
        \draw[color=black,thick] (\i*3/4,\j) -- (\i*3/4+3/4,\j);
      }
      \draw[color=gray,dotted,thick] (3,0) -- (3,1);
    \end{tikzpicture}
    \\[3ex]
    \begin{tikzpicture}[scale=1.5]
      \draw[domain=0:1,color=blue!50] plot (\x,{\x*\x});
      \draw[domain=1:1.4,color=blue!50] plot (\x,{0.8});
      \draw[domain=1.4:2,color=blue!50] plot (\x,{0.6});    
      \draw[domain=2:3,color=blue!50] plot (\x,{0.5+\x/7*sin(2*3.14159*(\x-1.375) r))});
      \foreach \i / \j in {0/0.0469,1/0.3281,2/0.7806,3/0.7467,4/0.6000,5/0.3511,6/0.6129,7/0.6121} {
        \draw[color=gray,dotted,thick] (\i*3/8,0) -- (\i*3/8,1);
        \draw[color=black,thick] (\i*3/8,\j) -- (\i*3/8+3/8,\j);
      }
      \draw[color=gray,dotted,thick] (3,0) -- (3,1);
    \end{tikzpicture}
    \qquad
    \begin{tikzpicture}[scale=1.5]
      \draw[domain=0:1,color=blue!50] plot (\x,{\x*\x});
      \draw[domain=1:1.4,color=blue!50] plot (\x,{0.8});
      \draw[domain=1.4:2,color=blue!50] plot (\x,{0.6});    
      \draw[domain=2:3,color=blue!50] plot (\x,{0.5+\x/7*sin(2*3.14159*(\x-1.375) r))});
      \foreach \i / \j in {0/0.0117,1/0.0820,2/0.2227,3/0.4336,4/0.7148,5/0.8463,6/0.8000,7/0.6933,8/0.6000,9/0.6000,10/0.4867,11/0.2155,12/0.4409,13/0.7850,14/0.8032,15/0.4211} {
        \draw[color=gray,dotted,thick] (\i*3/16,0) -- (\i*3/16,1);
        \draw[color=black,thick] (\i*3/16,\j) -- (\i*3/16+3/16,\j);
      }
      \draw[color=gray,dotted,thick] (3,0) -- (3,1);
    \end{tikzpicture}
  \end{center}
  \caption{Approximation of a function of bounded variation by its
    piecewise averages according to Lemma~\ref{lem:approx}: a given
    function~$u\in BV_{\star}(0,T)$ (upper left) and the piecewise averages
    $u_{[4]}$, $u_{[8]}$, and~$u_{[16]}$.}
  \label{fig:app}
\end{figure}
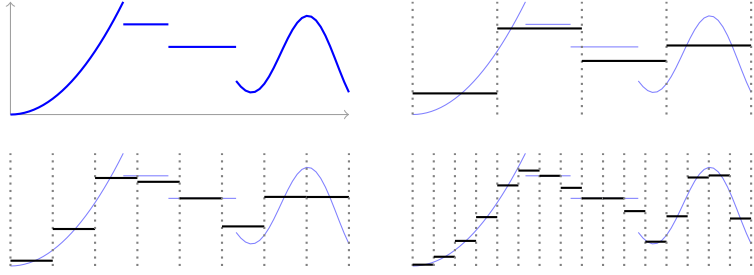
\begin{lemma}\label{lem:approx}
  Let~$u\in BV_{\star}(0,T)$. Then the piecewise averages~$\avn uN$
  strongly converge to~$u$ in~$L^2_{\star}(0,T)$ for~$N\to\infty$.
\end{lemma}
\begin{proof}
  Given $N\in\N$, we first claim that
    $||u-(u_N)_i||_{L^\infty(I_i)}\le |u|_{BV(I_i)}$ for all~$i=1,\dots,N$,
    where we set~$I_i:=\big((i-1)\frac TN,i\frac TN)$.
    Indeed, for almost all $t\in I_i$ we have
    \begin{eqnarray*}
      |u(t)-(u_N)_i|
      & = & \tfrac NT\, \Big|\int_{(i-1)\frac TN}^t\underbrace{u(t)-u(\tau)}_{=\partial u\left((\tau,t]\right)}\,\d\tau+
    \int_{t}^{i\frac TN} \underbrace{u(t)-u(\tau)}_{=-\partial u\left((t,\tau]\right)}\,\d\tau
    \Big|\\
    & \le & \tfrac NT\, \Big(\int_{(i-1)\frac
      TN}^t\underbrace{|\partial u|\big((\tau,t]\big)}_{\le |\partial u|(I_i)}\,\d\tau+
    \int_{t}^{i\frac TN} \underbrace{|\partial u|\big((t,\tau]\big)}_{\le|\partial u|(I_i)}\,\d\tau
    \Big)\\ 
    & \le & |\partial u|(I_i) = |u|_{BV(I_i)}\;,
    \end{eqnarray*}
    where $\partial u\big((t_1,t_2]\big)=\partial
      u\big([0,t_2]\big)-\partial u\big(([0,t_1]\big)=u(t_2)-u(t_1)$
      for almost all~$t_1,t_2\in I_i$, $t_1<t_2$ using the
      representative~$t\mapsto \partial u\big([0,t]\big)$ of~$u$
      mentioned above.  We now derive
  \begin{eqnarray*}
    ||u-\avn uN||_{L^2(0,T)}^2 & = & \int_0^T |u(t)-\avn uN(t)|^2\,\d t\\
    & = & \sum_{i=1}^N\int_{I_i} |u(t)-(u_N)_i|^2\,\d t\\
    & \le & \sum_{i=1}^N \tfrac TN ||u-(u_N)_i||_{L^\infty(I_i)}^2\\
    & \le & \sum_{i=1}^N\tfrac TN |u|_{BV(I_i)}^2
    ~ \le ~ \tfrac TN|u|_{BV(-\infty,T)}^2\;,
  \end{eqnarray*}
  and the latter expression converges to zero for~$N\to\infty$
  since~$|u|_{BV(-\infty,T)}$ is finite. The first inequality in this
  chain follows from Hölder's inequality, the second from the claim
  shown above, and the last one from the fact that the intervals~$I_i$
  form a partition of~$(0,T)$.
\end{proof}
  
We finally note that, if~$Du\ge 0$, the monotonicity of~$u$ is also
reflected in the discretization:
\begin{lemma}\label{lem:mono}
  Let~$u\in L^2_{\star}(0,T)$ with~$Du\ge 0$. Then $(u_N)_N\ge
  (u_N)_{N-1}\ge\dots\ge (u_N)_1\ge 0$.
\end{lemma}
\begin{proof}
  For~$i=0,\dots,N$, define~$\varphi_i\colon\R\to\R$ by
  \[
  \varphi_i(t)=\begin{cases}
  \begin{array}{ll}
    0 & \text{for }t\in (-\infty,(i-1)\tfrac TN]\;,\\ t-(i-1)\tfrac TN
      & \text{for }t\in [(i-1)\tfrac TN,i\tfrac TN]\;,\\ \tfrac TN &
      \text{for }t\in[i\tfrac TN,\infty)\;.
  \end{array}
  \end{cases}
  \]
  Then, for~$i=1,\dots,N$, the function~$\varphi_{i-1}-\varphi_{i}$ is
  non-negative and belongs to~$H^1_0(-\infty,T)$, since it vanishes
  at~$-\tfrac TN$ and~$T$.  Using $Du\ge 0$,
  Lemma~\ref{lem:approx_cont} now yields
  \[
  \int_{-\infty}^T(\varphi_i'(t)-\varphi_{i-1}'(t))u(t)\,\d t\ge 0\,.
  \]
  This implies
  \[
  (u_N)_i = \tfrac NT\int_{-\infty}^T \varphi'_i(t)u(t)\,\d t \ge
  \tfrac NT\int_{-\infty}^N \varphi'_{i-1}(t)u(t)\,\d t
  \]
  and the latter integral is~$(u_N)_{i-1}$ for~$i=2,\dots,N$ and zero
  for~$i=1$.
\end{proof}

\section{Extended Formulations}\label{sec:extform}

Our aim is to devise extended formulations for binary switching
constraints on a continuous time horizon~$[0,T]$. More precisely, we
investigate specific subsets~$U\subseteq BV_{\star}(0,T;\{0,1\})$ that are
bounded in~$BV_{\star}(0,T;\{0,1\})$, meaning that there exists a uniform
bound~$\sigma\in\Q_+$ such that~$|u|_{BV(-\infty,T)}\le\sigma$ for
all~$u\in U$. Using the lower semi\-continuity
of~$|\cdot|_{BV(-\infty,T)}$, one can show that this assumption
guarantees that~$\overline{\conv}(U)$ is still contained
in~$BV_{\star}(0,T)$. Here and in the following, all closures are taken
in~$L^2_{\star}(0,T)$.
\begin{definition}\label{def:extform}
  An \emph{extended formulation} of~$U\subseteq BV_{\star}(0,T)$ is a
  set~$U^\dext\subseteq L^2_{\star}(0,T)^{d+1}$ satisfying the
  following conditions:
  \begin{itemize}
  \item[(e1)] The projection of~$U^\dext$ to the first
    coordinate agrees with~$\cconv(U)$.
  \item[(e2)] The formulation is \emph{linear}, i.e., the
    elements~$(u,z_1,\dots,z_d)\in U^\dext$ can be characterized
    by finitely many constraints, each of which is either of the form
    \[
    L(u,z_1,\dots,z_d)\ge 0\text{ \aei\ }(0,T)
    \]
    or of the form
    \[
    DL(u,z_1,\dots,z_d)\ge 0
    \]
    for a continuous affine-linear
    operator $L\colon L^2_{\star}(0,T)^{d+1}\to L^2_{\star}(0,T)$.
  \end{itemize}
  The extended formulation is called \emph{small} if the number of
  controls~$d+1$ as well as the number of constraints
  describing~$U^\dext$ according to~(e2) are polynomial in the input size.
\end{definition}
Herein, the projection of~$U^\dext$ to the first coordinate is defined
as the set
\[
\{u\in L^2_{\star}(0,T)\mid \exists z_1,\dots,z_d\in
L^2_{\star}(0,T)\text{ s.t.\ }(u,z_1,\dots,z_d)\in U^\dext\}\;.
\]
Polynomiality in the input size is a very common requirement in
combinatorial optimization. Nonetheless, in the infinite-dimensional
setting, it calls for some clarification. Obviously, the concept is
only well-defined when the sets~$U$ under consideration belong to a
class that can be finitely parametrized, i.e., an individual
instance~$U$ from the class is fully determined by a finite input.  To
emphasize this, all input parameters except for~$T$ will be part of
the notation for all problem classes investigated in the subsequent
sections.

In finite dimension, extended formulations of polynomial size are
usually called ``compact''. However, in our context, this term would
be ambiguous, since it could also refer to the compactness of the
set~$U^\dext$ in the Banach space~$L^2_{\star}(0,T)^{d+1}$. For this
reason, we use the notion of smallness instead of compactness
throughout this paper.

The next result shows that
every extended formulation is closed in~$L^2_{\star}(0,T)^{d+1}$. Moreover,
the projection in Definition~\ref{def:extform} is always closed if the
extended formulation is bounded.
\begin{lemma}\label{lem:convclosed}
  Every extended formulation~$U^\textnormal{ext}$ is convex and closed
  in~$L^2_{\star}(0,T)^{d+1}$. Moreover, if~$U^\dext$ is bounded
  in~$L^2_{\star}(0,T)^{d+1}$, then the projection of~$U^\dext$ to the first
  coordinate is convex and closed in~$L^2_{\star}(0,T)$.
\end{lemma}
\begin{proof}
  Clearly, all constraints described in Condition~(e2) of
  Definition~\ref{def:extform} are convex. To show closedness, let
  $L\colon L^2_{\star}(0,T)^{d+1}\to L^2_{\star}(0,T)$ be any continuous and
  affine-linear operator. Then constraints of
  type~$L(u,z_1,\dots,z_d)\ge 0$ are obviously closed, while
  constraints of type~$DL(u,z_1,\dots,z_d)\ge 0$ are closed by
  Lemma~\ref{lem:nndu}.

  Now assume that~$U^\dext$ is bounded in~$L^2_{\star}(0,T)^{d+1}$ and
  let~$U'$ be the projection of~$U^\dext$ to the first
  coordinate. Then clearly~$U'$ is convex. To show closedness,
  let~$u_n$ be a sequence in~$U'$ that strongly converges to
  some~$u\in L^2_{\star}(0,T)$. For~$n\in\N$,
  choose~$(z_1)_n,\dots,(z_d)_n$ in~$L^2_{\star}(0,T)$
  with~$(u_n,(z_1)_n,\dots,(z_d)_n)\in U^\dext$.  Since~$U^\dext$ is
  bounded in the reflexive Banach space~$L^2_{\star}(0,T)^{d+1}$,
  there exists a subsequence of~$(u_n,(z_1)_n,\dots,(z_d)_n)$
  converging weakly to some~$(u',z_1,\dots,z_d)\in U^\dext$. But~$u_n$
  converges strongly to~$u$ by assumption, hence we must have~$u'=u$,
  so that~$u$ belongs to~$U'$. Thus $U'$ is closed
  in~$L^2_{\star}(0,T)$.
\end{proof}

In finite dimension, the main relevance of small extended
formulations stems from the fact that they provide efficient and elegant
approaches to solving combinatorial optimization problems, by
reformulating them as polynomial-size linear programs. In order to
derive a similar result in function space, it is necessary that the
extended formulations carry over to the discretization of the
problem. This is formalized in the following definition.
\begin{definition}\label{def:compat}
  An extended formulation~$U^\dext$ of~$U$ is \emph{compatible with
  discretization} if for all~$N\in\N$ the following conditions hold:
  \begin{itemize}
  \item[(c1)] the projection of~$(U^\dext)_N$ to the first coordinate
    agrees with~$\conv(U_N)$,
  \item[(c2)] the affine-linear operators $L$ defining~$U^\dext$
    map~$L^2_{\star}(0,T)^{d+1}_N$ to~$L^2_{\star}(0,T)_N$, and
  \item[(c3)] the coefficients of the restrictions
    $L|_{L^2_{\star}(0,T)^{d+1}_N}$ can be computed in polynomial
    time.
  \end{itemize}
\end{definition}
Note that the restrictions mentioned in (c3)
map~$L^2_{\star}(0,T)^{d+1}_N$ to $L^2_{\star}(0,T)_N$ according
to~(c2), so that they can be identified with affine-linear
functions~$\R^{(d+1)N}\to\R^N$, which are defined by polynomially many
coefficients.

Using Definition~\ref{def:compat}, we now have
\begin{theorem}\label{thm:extended}
  Assume that~$U$ admits a small extended formulation~$U^{\dext}$
  that is compatible with discretization. Then, for fixed~$N\in\N$,
  the discretization~$U_N$ of~$U$ is tractable, i.e., for any~$c\in
  L^2_{\star}(0,T)$ given by its piecewise averages~$c_N$, the linear
  objective function
  \begin{equation}\label{eq:obj}
    \int_0^T c(t)u(t)\,\d t
  \end{equation}
  can be minimized over~$u\in U_N$ in polynomial time.
\end{theorem}
\begin{proof}
  Given~$(u,z_1,\dots,z_d)\in L^2_{\star}(0,T)_N$, we
  have~$(u,z_1,\dots,z_d)\in (U^\dext)_N$ if and only if the
  constraints in Condition~(e2) of Definition~\ref{def:extform} are
  satisfied by~$(u,z_1,\dots,z_n)$. By Condition~(c1),
  minimizing~\eqref{eq:obj} over~$u\in U_N$ or, equivalently,
  over~$u\in\conv(U_N)$, thus reduces to minimizing the same linear
  function over~$(u,z_1,\dots,z_d)\in (U^\dext)_N$.  Identifying
  elements of~$L^2_{\star}(0,T)_N$ with vectors in~$\R^N$, the latter in
  turn reduces to the minimization of~$c_N^\top u$ subject to
  constraints of type~$L(\overline
  u,\overline{z_1},\dots,\overline{z_d})\ge 0$ and~$DL(\overline
  u,\overline{z_1},\dots,\overline{z_d})\ge 0$, where the variables
  are~$u,z_1,\dots,z_d\in\R^N$. By the smallness of the extended
  formulation and by~(c2), the latter constraints translate to
  polynomially many affine-linear constraints in~$(u,z_1,\dots,z_d)$,
  the coefficients of which can be computed efficiently by
  Condition~(c3). In summary, the minimization problem reduces to the
  solution of a polynomial-size linear program in dimension~$(d+1)N$.
\end{proof}
Note that the running time for minimization over~$U_N$ obtained in the
above proof is only pseudopolynomial in~$N$, as~$N$ enters the
dimension of the resulting linear program.

Formally, the set~$(U^\dext)_N$ is \emph{not} an extended formulation
of~$U_N$ when interpreting elements of~$U_N$ and~$(U^\dext)_N$ as
piecewise constant functions rather than finite-dimensional vectors,
as in Definition~\ref{def:disc}. Indeed, the linear description
derived from~$U^\dext$ does not guarantee that the functions are
piecewise constant, i.e., it does not completely
describe~$(U^\dext)_N$.

Theorem~\ref{thm:extended} states that small and compatible extended
formulations immediately lead to tractable discretizations,
independent of the chosen number~$N$ of grid cells. However, depending
on the context, it may be enough from a practical point of view to
obtain a tractable discretization after a suitable refinement of the
grid. As we will see in the following, it may well happen that a
discretization turns tractable only after increasing the number of
grid cells. For this reason, we also consider the following weaker
condition:
\begin{definition}\label{def:weakextform}
  An extended formulation~$U^{\dext}$ of~$U$ is \emph{weakly
  compatible with discretization} if for all~$M\in\N$ one can
  efficiently compute some~$\ell\in\N$ which is polynomial in~$M$ such
  that Conditions~(c1) to~(c3) of Definition~\ref{def:compat} hold
  for~$N=\ell M$.
\end{definition}
To clarify the polynomiality condition in this definition, it is
required that~$\ell$ can be computed in polynomial time from the input
of the problem, i.e., the input defining the instance~$U$ within the
given class of feasible sets, and from~$M$. Moreover, the value
of~$\ell$ must be polynomial in the value of~$M$. This is crucial
since the dimension of the linear programs resulting from
discretization is linear in the number of grid cells. For a fixed
problem input, the size of the linear program thus increases only
polynomially by the refinement.

If~$U$ admits a small and weakly compatible extended formulation, it
follows directly from Theorem~\ref{thm:extended} that for
any~$M\in\N$ one can efficiently compute some~$\ell\in\N$ polynomial
in~$M$ such that~$U_{\ell M}$ is tractable. It depends on the given
setting which of the two concepts of compatibility is appropriate: if
the discretization can be arbitrarily refined by the user, a weakly
compatible extended formulation is enough to obtain tractable
discretizations.

\section{Fixed Bound on the Variation}\label{sec:bounded}

Our first goal is to devise an extended formulation for the case where
a fixed bound on the variation is the only restriction on the set of
feasible controls; this case has been investigated
in~\cite{KMS11,ZS20,partI}. We thus consider the set
\[
\DD:=\{u\in BV_{\star}(0,T;\{0,1\})\colon |u|_{BV(-\infty,T)}\le \sigma\}
\]
for given~$\sigma\in\N$. A natural convex relaxation of $\DD$ arises
from omitting the binarity constraint on the control, we then obtain
\[
\DD^{\textnormal{rel}} := \big\{ u \in BV_{\star}(0,T;[0,1])\colon \;
|u|_{BV(-\infty,T)} \leq \sigma \big\}\;.
\]
Note that the set~$\DD^{\textnormal{rel}}$ is convex and, by the lower
semicontinuity of~$|\cdot|_{BV(0,T)}$ and the boundedness
of~$\DD^{\textnormal{rel}}$ in~$BV(0,T)$, also closed
in~$L^2_{\star}(0,T)$. We now show that~$\DD^{\textnormal{rel}}$ can be
obtained as the projection of a very small extended formulation,
where a single additional function~$z$ keeps track of the accumulated
variation of~$u$.
\begin{theorem}\label{thm:nat}
  The set
  \[
  \DD^{\textnormal{rel,ext}}:=\big\{ u\in L^2_{\star}(0,T;[0,1]),~z\in
  L^2_{\star}(0,T;[0,\sigma])\colon Dz \ge Du,~Dz \ge -Du\big\}
  \]
  is a small extended formulation of~$\DD^{\textnormal{rel}}$
  which is compatible with discretization.
\end{theorem}
\begin{proof}
  In the format of Definition~\ref{def:extform}, the constraints in
  the above formulation explicitly read $u\ge 0$, $-u+\boldsymbol 1\ge
  0$, $z\ge 0$, $-z+\boldsymbol\sigma\ge 0$, $D(z-u)\ge 0$, and
  $D(z+u)\ge0 $, where~$\boldsymbol 1$ and~$\boldsymbol\sigma$ denote
  the functions being constantly~$1$ and~$\sigma$,
  respectively. Condition~(e2) of Definition~\ref{def:extform} as well
  as smallness are thus obvious, so it suffices to show
  Condition~(e1) and compatibility with discretization.

  We first show
  Condition~(e1). Given~$(u,z)\in\DD^{\textnormal{rel,ext}}$, we
  observe that~$z$ is essentially bounded by~$\sigma$ and
  monotonically increasing by~$Dz\ge 0$. Thus~$z\in BV_{\star}(0,T)$,
  such that~$Dz$ can be represented by a regular Borel
  measure~$\partial z$ satisfying~$\partial
  z\big((-\infty,T)\big)\le\sigma$. Now for
  every~$\varphi\in\testfunc{-\infty,T}$, we have
  $\varphi_{\max},\varphi_{\min}\in H^1_{0}(-\infty,T)$
  for~$\varphi_{\max}:=\max\{\varphi,0\}$
  and~$\varphi_{\min}:=-\min\{\varphi,0\}$ and the weak derivatives
  satisfy~$\varphi'=\varphi_{\max}'-\varphi_{\min}'$. Using~$Dz\ge Du$
  and~$Dz\ge -Du$ together with Lemma~\ref{lem:approx_cont}, we derive
  \begin{eqnarray*}
    \int_{-\infty}^T\varphi'(t)u(t)\,\d t
    & = & -\int_{-\infty}^T\varphi_{\max}'(t)(-u(t))\,\d t-\int_{-\infty}^T\varphi_{\min}'(t)u(t)\,\d t\\
    & \le & -\int_{-\infty}^T\varphi_{\max}'(t)z(t)\,\d t-\int_{-\infty}^T\varphi_{\min}'(t)z(t))\,\d t\\ 
    & = & \int_{-\infty}^T\varphi_{\max}\,\d \partial z+\int_{-\infty}^T\varphi_{\min}\,\d \partial z\\
    & = & \int_{-\infty}^T|\varphi|\,\d \partial z ~\le~ \sigma\, ||\varphi||_\infty\;.
  \end{eqnarray*}
  Since~$\varphi$ was arbitrary, this implies~$u\in BV_{\star}(0,T)$
  with~$|u|_{BV(-\infty,T)}\le \sigma$,
  thus~$u\in\DD^{\textnormal{rel}}$.  For the other direction,
  let~$u\in\DD^{\textnormal{rel}}$ and define~$z\in L^2_{\star}(0,T)$
  by $z(t):=|\partial u|\big((-\infty,t)\big)$ for~$t\in[0,T]$.  Then
  clearly~$(u,z)\in\DD^{\textnormal{rel,ext}}$ and hence~$u$ belongs
  to the projection of~$\DD^{\textnormal{rel,ext}}$ to the first
  coordinate.

  To show compatibility with discretization, first consider
  Condition~(c1) of Definition~\ref{def:compat} and let~$u\in
  (\DD^{\textnormal{rel}})_N\subseteq\DD^{\textnormal{rel}}$. Defining
  $z(t)=|\partial u|\big((-\infty,t)\big)$ again, we obtain that~$z$
  is piecewise constant on the same grid as~$u$. Thus~$z\in
  L^2_{\star}(0,T)_N$ and~$(u,z)\in (\DD^{\textnormal{rel,ext}})_N$,
  hence~$u$ belongs to the projection
  of~$(\DD^{\textnormal{rel,ext}})_N$ to the first component. For the
  reverse inclusion, let~$u$ belong to the latter
  projection. Choose~$z\in L^2_{\star}(0,T)_N$ with~$(u,z)\in
  (\DD^{\textnormal{rel,ext}})_N$. Then~$u\in \DD^{\textnormal{rel}}$
  follows from Condition~(e1) already shown above and~$u\in
  L^2_{\star}(0,T)_N$ follows from the choice of~$u$, thus~$u\in
  \DD^{\textnormal{rel}}_N$. This shows~(c1).

  Finally, the remaining conditions~(c2) and~(c3) of
  Definition~\ref{def:compat} are obviously satisfied: given~$N\in\N$,
  the discretized constraints read $u_i\ge 0$, $-u_i+1\ge 0$, $z_i\ge
  0$, $-z_i+\sigma\ge 0$, $(z_i-u_i)-(z_{i-1}-u_{i-1})\ge 0$,
  and~$(z_i+u_i)-(z_{i-1}+u_{i-1})\ge 0$ for~$i=1,\dots,N$.
\end{proof}
By Theorem~\ref{thm:nat}, the projection
of~$\DD^{\textnormal{rel,ext}}$ to the first coordinate yields a
relaxation for the set~$\DD$. However, it can be shown that
$\cconv(\DD)$ is \emph{strictly} contained in~$\DD^{\textnormal{rel}}$
in general; see~\cite[Counterexample~3.1]{partI}. In other words, the
formulation in Theorem~\ref{thm:nat} is \emph{not} an extended
formulation for~$\DD$ in general. Our next goal is thus to devise an
extended formulation that instead projects to~$\overline{\conv}(\DD)$.

We construct such a formulation by first considering the discretized
problem, for which an extended formulation has been presented
in~\cite[Section~4.5]{maja}: assuming that~$\sigma$ is even, the
projection of
\[
\begin{aligned}
  (\DD_N)^{\dext} := \big\{ & u \in [0,1]^N,~z\in\R^N \colon\\
  & u_i-u_{i-1}\le z_i-z_{i-1} \quad\text{for }i=1,\dots,N,\\
  & 0 \le z_1 \le z_2 \le \dots \le  z_N \le \tfrac{\sigma}2\big\}
\end{aligned}
\]
to the $u$-space coincides with~$\conv(\DD_N)$, where again we
identify piecewise constant functions with finite-dimensional
vectors. This follows from~\cite[Lemma~4.29]{maja} using the
substitution~$z_i\leftarrow\sum_{j=1}^iz_j$. In this model, we
define~$z_0=u_0=0$, but do not eliminate these variables
for sake of a simpler formulation. The first constraint thus reduces
to~$u_1\le z_1$.

We now turn our attention back to the infinite-dimensional setting and
show that~$\cconv(\DD)$ can be described by an extended formulation
that is inspired by the formulation~$(\DD_{N})^{\dext}$. For this,
define
\[
\DD^{\dext}:=\big\{ u\in L^2_{\star}(0,T;[0,1]),~z\in
L^2_{\star}(0,T;[0,\tfrac{\sigma}2])\colon Dz\ge Du,~Dz\ge 0\big\}\;.
\]
For an illustration of this construction, see
Figure~\ref{fig:ext}. Using this, we now obtain an extended
formulation of~$\DD$.

\begin{figure}
  \begin{center}
    \begin{tikzpicture}[scale=1.3]
      \draw[->,color=gray!75] (0,0) -- (0,2);
      \draw[->,color=gray!75] (0,0) -- (4,0);
      \draw[color=gray!75,dashed] (0,0.5) -- (4,0.5);
      \draw[color=gray!75] (3.5,0.1) -- (3.5,-0.1);
      \draw (-0.2,0) node {\small $0$};
      \draw (-0.2,0.5) node {\small $1$};
      \draw (3.7,0.2) node {\small $T$};
      \draw[color=blue] (0.3,0.7) node {\small $u$};
      \draw[color=blue,thick] (0,0) -- (0.3,0);
      \draw[color=blue,dotted,thick] (0.3,0) -- (0.3,0.5);
      \draw[color=blue,thick] (0.3,0.5) -- (0.9,0.5);
      \draw[color=blue,dotted,thick] (0.9,0.5) -- (0.9,0);
      \draw[color=blue,thick] (0.9,0) -- (1.5,0);
      \draw[color=blue,dotted,thick] (1.5,0) -- (1.5,0.5);
      \draw[color=blue,thick] (1.5,0.5) -- (1.8,0.5);
      \draw[color=blue,dotted,thick] (1.8,0.5) -- (1.8,0);
      \draw[color=blue,thick] (1.8,0) -- (2.1,0);
      \draw[color=blue,dotted,thick] (2.1,0) -- (2.1,0.5);
      \draw[color=blue,thick] (2.1,0.5) -- (3.0,0.5);
      \draw[color=blue,dotted,thick] (3.0,0.5) -- (3.0,0);
      \draw[color=blue,thick] (3.0,0) -- (3.5,0);
    \end{tikzpicture}
    \qquad
    \begin{tikzpicture}[scale=1.3]
      \draw[->,color=gray!75] (0,0) -- (0,2);
      \draw[->,color=gray!75] (0,0) -- (4,0);
      \draw[color=gray!75,dashed] (0,0.5) -- (4,0.5);
      \draw[color=gray!75,dashed] (0,1.5) -- (4,1.5);
      \draw[color=gray!75] (3.5,0.1) -- (3.5,-0.1);
      \draw (-0.2,0) node {\small $0$};
      \draw (-0.2,1.5) node {\small $3$};
      \draw (3.7,0.2) node {\small $T$};
      \draw[color=red] (0.375,0.75) node {\small $u^+$};
      \draw[color=blue] (0,0) -- (0.3,0);
      \draw[color=blue,dotted] (0.3,0) -- (0.3,0.5);
      \draw[color=blue] (0.3,0.5) -- (0.9,0.5);
      \draw[color=blue,dotted] (0.9,0.5) -- (0.9,0);
      \draw[color=blue] (0.9,0) -- (1.5,0);
      \draw[color=blue,dotted] (1.5,0) -- (1.5,0.5);
      \draw[color=blue] (1.5,0.5) -- (1.8,0.5);
      \draw[color=blue,dotted] (1.8,0.5) -- (1.8,0);
      \draw[color=blue] (1.8,0) -- (2.1,0);
      \draw[color=blue,dotted] (2.1,0) -- (2.1,0.5);
      \draw[color=blue] (2.1,0.5) -- (3.0,0.5);
      \draw[color=blue,dotted] (3.0,0.5) -- (3.0,0);
      \draw[color=blue] (3.0,0) -- (3.5,0);
      \draw[color=red,thick] (0,0) -- (0.3,0);
      \draw[color=red,dotted,thick] (0.3,0) -- (0.3,0.5);
      \draw[color=red,thick] (0.3,0.5) -- (1.5,0.5);
      \draw[color=red,dotted,thick] (1.5,0.5) -- (1.5,1);
      \draw[color=red,thick] (1.5,1.0) -- (2.1,1);
      \draw[color=red,dotted,thick] (2.1,1) -- (2.1,1.5);
      \draw[color=red,thick] (2.1,1.5) -- (3.5,1.5);
    \end{tikzpicture}
  \end{center}
  \caption{Illustration of the extended formulation~$\DD^\dext$
    for~$\sigma=6$. A feasible control~$u$ is shown on the left hand
    side, the control~$z:=u^+$ is added on the right hand
    side.}
  \label{fig:ext}
\end{figure}
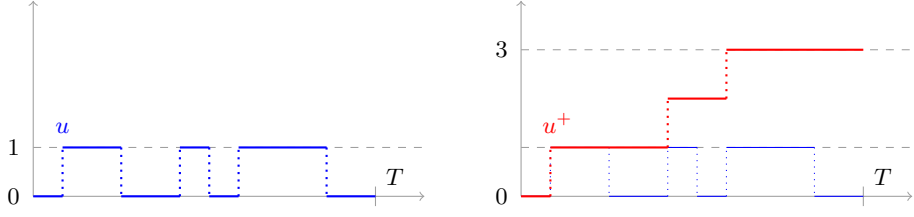

\begin{theorem}\label{thm:ext2}
  Let~$\sigma$ be even. Then~$\DD^{\dext}$ is a small
  extended formulation of~$\DD$ which is compatible with discretization.
\end{theorem}
\begin{proof}
  Again, Condition~(e2) of Definition~\ref{def:extform} and
  smallness are easily verified, so it suffices to show
  Condition~(e1) as well as compatibility. First assume that~$(u,z)\in
  \DD^{\dext}$. By Lemma~\ref{lem:approx}, we then have $(\avn uN,\avn
  zN)\to (u,z)$ for~$N\to\infty$. By construction, it follows
  that~$u_{N}\in[0,1]^{N}$ and~$z_{N}\in [0,\tfrac\sigma 2]^{N}$ for
  all~$N\in\N$. Moreover, applying Lemma~\ref{lem:mono} to both~$z$
  and~$z-u$ yields~$(z_{N})_i-(z_{N})_{i-1}\ge 0$
  and~$(z_{N})_i-(z_{N})_{i-1}\ge (u_{N})_{i}-(u_{N})_{i-1}$. In
  summary, we thus have $(u_N,z_N)\in (\DD_{N})^{\dext}$.
  Hence~$u_N\in\conv(\DD_N)$ and thus~$\avn uN\in\conv(\DD)$, using
  the finite-dimensional extended formulation. This shows~$u\in
  \overline\conv(\DD)$.

  Now let~$u\in\cconv(\DD)$. For showing that~$u$ belongs to the
  projection of~$\DD^{\dext}$ to the first coordinate, we
  may assume~$u\in\DD$ by Lemma~\ref{lem:convclosed}.
  Define~$z:=u^+$. Then~$z\in L^2_{\star}(0,T)$ with~$Dz\ge Du$ and~$Dz\ge
  0$, thus also~$z\ge 0$ almost
  everywhere. Moreover,
  \[
  \partial u^+\big([0,T)\big)+\partial u^-\big([0,T)\big)\le|u|_{BV(-\infty,T)}\le \sigma
  \]
  and, since~$u\le 1$ almost everywhere, we have
  \[
  \partial u^+\big([0,T)\big)-\partial u^-\big([0,T)\big)=\partial u\big([0,T)\big)\le 1\;.
  \]
  Summing up, we derive~$\partial z\big([0,T)\big)=\partial u^+\big([0,T)\big)\le\tfrac{\sigma+1}
  2$. Since~$\partial z\big([0,T)\big)$ is integer and~$\sigma$ is even, this
  implies~$\partial z\big([0,T)\big)\le\tfrac\sigma 2$ and hence~$z\le\tfrac\sigma 2$
  almost everywhere.

  Finally, Condition~(c1) of Definition~\ref{def:compat} follows
  directly from the fact that~$(\DD^\dext)_N$ corresponds to the
  finite-dimensional extended formulation~$(\DD_N)^\dext$ of~$\DD_N$,
  while~(c2) and~(c3) are again obvious.
\end{proof}
The preceding proof is a blueprint for similar results on extended
formulations in function space, provided that an extended formulation
for the discretized problem is known. The compatibility with
discretization follows easily in this situation. The first
inclusion of~(e1) is shown by approximating an element of the
supposed extended formulation by piecewise constant functions according
to Lemma~\ref{lem:approx} and then using the finite-dimensional
result, while for the other inclusion an explicit construction
of~$z_1,\dots,z_d$ from~$u$ is proposed.

Note that the exact extended formulation~$\DD^{\dext}$
of~$\DD$ is defined by the same number of constraints as the extended
formulation~$\DD^{\textnormal{ext,rel}}$ of the
relaxation~$\DD^{\textnormal{rel}}$. The difference between the
two formulations is that the exact formulation only counts the
positive variation and bounds it by~$\tfrac\sigma 2$, while the
relaxed formulation bounds the total variation by~$\sigma$.

In the case of odd~$\sigma$, one can use~\cite[Lemma~4.30]{maja} to
show that an extended formulation of~$\DD$ is given by
\[
\DD^{\dext}:=\big\{ u\in L^2_{\star}(0,T;[0,1]),~z\in
L^2_{\star}(0,T;[0,\tfrac{\sigma-1}2])\colon Dz\ge -Du,~Dz\ge 0\big\}\;.
\]
The proof is very similar to the proof of
Theorem~\ref{thm:ext2}. Instead of bounding the positive variation
by~$\tfrac\sigma 2$, as in the case of even~$\sigma$, we now bound the
negative variation by~$\tfrac{\sigma-1}2$.

\section{Minimum Dwell-time Constraints}\label{sec:dwell}

We next consider the important class of minimum dwell-time
constraints, which are also called min-up/min-down constraints in the
literature; see~\cite{lee04,RajanTakriti,BFR18} for the
finite-dimensional and~\cite{ZRS20} for the infinite-dimensional
case. Given~$L,l\in\Q_+$ with~$L+l>0$ as input, the feasible set~$\DM$
now consists of all~$u \in BV_{\star}(0,T;\{0,1\})$ such that, when
switching up~$u$ at time~$\tau$, no switching down occurs
in~$[\tau,\tau+L)$, and analogously, when switching down~$u$
  at~$\tau$, no switching up occurs in~$[\tau,\tau+l)$. Rajan and
    Takriti~\cite{RajanTakriti} have devised the following extended
    formulation for~$\DM_N$ for the case~$L,l\in\N$, where we have
    applied the same substitution of $z$-variables as in the previous
    section:
\[
\begin{aligned}
  (\DM_{N})^{\dext}:=\big\{ & u \in [0,1]^N,~z\in\R_+^N \colon\\
  & z_i-z_{i-L}\le u_i && \text{for }i=L+1,\dots,N,\\
  & z_i-z_{i-l}\le 1-u_{i-l} && \text{for }i=l+1,\dots,N,\\
  & z_i-z_{i-1} \ge u_i-u_{i-1} && \text{for }i=1,\dots,N,\\
  & z_i-z_{i-1} \ge 0 && \text{for }i=1,\dots,N\big\}\;.
\end{aligned}
\]
Here~$z_0$ and~$u_0$ are again defined as zero. In order to exploit
this model for general~$L,l\in\Q$, the number of grid cells~$N$ must
be chosen such that $\nicefrac{NL}T,\nicefrac{Nl}T\in\N$, as we will
see in the proof of Theorem~\ref{thm:eextended} below (which in
practice represents a fairly severe restriction).

To obtain an extended formulation for the infinite-dimensional
set~$\overline\conv(\DM)$, we first define a continuous linear map
\[
V_r\colon L^2_{\star}(0,T)\to L^2_{\star}(0,T)
\]
for given~$r\in\R_+$ by $V_r(z)(t)=z(t-r)$, i.e., $V_r$ shifts~$z$ to
the right by~$r$. Note that, by the definition of~$L^2_{\star}(0,T)$, this
implies~$V_r(z)=0$ \aei~$(-\infty,r)$. Now consider the following
formulation, inspired by the extended formulation of Rajan and Takriti:
\begin{subequations}
\begin{align}
  \DM^{\dext}:=\big\{ & u\in L^2_{\star}(0,T;[0,1]),~z\in L^2_{\star}(0,T;[0,\sigma])\colon\notag\\
  & z-V_{L}(z)\le u,\label{eq:v1}\\
  & z-V_{l}(z)\le {\bf 1}-V_{l}(u),\label{eq:v2}\\
  & Dz\ge Du,~Dz\ge 0\big\}\;.\label{eq:v3}
\end{align}
\end{subequations}
Here we define~$\sigma:=\lceil\nicefrac {2T}{L+l}\rceil$.  This
value of~$\sigma$ is large enough to make the upper bound on~$z$
redundant, it is only used to ensure that~$\DM^{\dext}$ is
bounded in~$L^2_{\star}(0,T)^2$.

\pagebreak[3]

\begin{theorem}\label{thm:eextended}
  The set~$\DM^{\dext}$ is a small extended formulation
  of~$\DM$ which is weakly compatible with discretization.
\end{theorem}
\begin{proof}
  Again, Condition~(e2) of Definition~\ref{def:extform} as well as
  smallness are easily verified, so it suffices to show
  Condition~(e1) and weak compatibility. So let~$(u,z)\in
  \DM^{\dext}$. We again have
  \[
  (\avn uN,\avn zN)\to (u,z)\text{ for }N\to\infty
  \]
  by Lemma~\ref{lem:approx}.  Now choose~$\ell\in\N$ such
  that~$\nicefrac{\ell L}T$ and~$\nicefrac{\ell l}T$ are both integer
  and consider the subsequence given by indices~$N_k:=k\ell$,
  $k\in\N$. We claim that
  \begin{equation}\label{eq:inext}
    (u_{N_k},z_{N_k})\in
    (U(\nicefrac{N_kL}T,\nicefrac{N_kl}T)_{N_k})^{\dext}~\text{
      for all }k\in\N\;.
  \end{equation}
  First note that~$u_{N_k}\in[0,1]^{N_k}$ and $z_{N_k}\in\R_+^{N_k}$
  by construction. Moreover, Lemma~\ref{lem:mono} again
  yields~$(z_{N_k})_i-(z_{N_k})_{i-1}\ge (u_{N_k})_i-(u_{N_k})_{i-1}$
  and~$(z_{N_k})_i-(z_{N_k})_{i-1}\ge 0$ using~\eqref{eq:v3}. Finally,
  \begin{eqnarray*}
    (z_{N_k})_i-(z_{N_k})_{i-\nicefrac {N_kL}T} & = & \tfrac {N_k}T\int_{(i-1)\frac T{N_k}}^{i\frac T{N_k}} z(t)\,\d t-\tfrac {N_k}T\int_{(i-\nicefrac{N_kL}T-1)\frac T{N_k}}^{(i-\nicefrac{N_kL}T)\frac T{N_k}} z(t)\,\d t\\
    & = & \tfrac {N_k}T\int_{(i-1)\frac T{N_k}}^{i\frac T{N_k}} z(t)\,\d t-\tfrac {N_k}T\int_{(i-1)\frac T{N_k}}^{i\frac T{N_k}} V_{L}(z)(t)\,\d t\\
    & = & \tfrac {N_k}T\int_{(i-1)\frac T{N_k}}^{i\frac T{N_k}} \big(z(t)-V_L(z)(t)\big)\,\d t\\
    & \underset{\eqref{eq:v1}}{\le} & \tfrac {N_k}T\int_{(i-1)\frac T{N_k}}^{i\frac T{N_k}} u(t)\,\d t=(u_{N_k})_i\;,
  \end{eqnarray*}
  and the remaining constraint can be shown analogously
  using~\eqref{eq:v2}. This concludes the proof
  of~\eqref{eq:inext}. The result of Rajan and Takriti now
  shows~$u_{N_k}\in\conv(U(\nicefrac{N_kL}T,\nicefrac{N_kl}T)_{N_k})$
  for all~$k\in\N$, yielding~$\avn u{N_k}\in\conv(\DM)$, which
  implies~$u\in \overline\conv(\DM)$.

  For showing the other inclusion, we may assume~$u\in\DM$ by
  Lemma~\ref{lem:convclosed}. Define~$z:=u^+$. Then we have~$Dz\ge Du$
  and~$Dz\ge 0$. Moreover, the linear inequality~\eqref{eq:v1} follows
  from the definition of~$\DM$ then. Indeed, we
  obtain~$z-V_L(z)=u^+-V_L(u^+)\in\{0,1\}$ a.e.\ from the definition
  of~$\DM$, since~$u$ can switch up at most once in any time interval
  of length~$L$.  If~$u^+-V_L(u^+)=1$ \aei\ some
  interval~$(\tau_1,\tau_2)$, we derive that~$\partial u^+(\{\tau\})=1$ for
  some~$\tau\in(\tau_2-L,\tau_1)$. From the definition of~$\DM$, we
  obtain~$u=1$ \aei~$(\tau_1,\tau_2)$, showing~\eqref{eq:v1}. The
  constraint~\eqref{eq:v2} follows by a similar
  reasoning. Thus~$(u,z)$ satisfies all conditions of~$\DM^{\dext}$,
  which shows that~$u$ belongs to the projection of~$\DM^{\dext}$ to
  the first coordinate.

  From the proof so far, it follows that the extended
  formulation~$\DM^\dext$ is compatible with discretization provided
  that~$N$ is a multiple of~$\ell$. A feasible value for~$\ell$ can be
  computed efficiently from the rational numbers~$L$, $l$, and~$T$,
  e.g., by multiplying the numerator of~$T$ with the denominators
  of~$L$ and~$l$.  We thus obtain that the extended formulation is
  weakly compatible with discretization.
\end{proof}
As seen in the proof of Theorem~\ref{thm:eextended}, the refinement
factor~$\ell$ required to obtain a compatible discretization does not
depend on~$N$, but on the input consisting of~$L$, $\l$, and~$T$. Even
when fixing~$N$, the number of necessary grid cells thus grows by a
factor that is polynomial in these numbers. In terms of~$L$, $\l$,
and~$T$, we thus only obtain a pseudopolynomial algorithm for solving
the refined discretizations. However, at least for the extended
formulation~$\DM^\dext$ considered here, this is unavoidable, since
even Condition~(c2) of Definition~\ref{def:compat} is not satisfied
unless~$\nicefrac TN$ is an integer multiple of~$L$ and~$l$.

\section{Linear Switching Point Constraints}\label{sec:general}

The minimum dwell-time constraints considered in the previous section
form a special case of linear switching point constraints.  Since we
assume~$U$ to be bounded in~$BV_{\star}(0,T;\{0,1\})$, each function~$u\in
U$ is defined by its finitely many switching
points~$t_1,\dots,t_\sigma$, i.e., the points where the value of~$u$
changes; see~\cite[Section~3.2]{partI} for a formal definition. We
assume again that~$u$ starts being switched off and consider the
set~$U(A,b)\subseteq BV_{\star}(0,T;\{0,1\})$ consisting of all
functions~$u$ such that its switching points satisfy given linear
inequalities~$At\le b$ for~$A\in\Q^{m\times\sigma}$
and~$b\in\Q^m$. More precisely, taking into account that functions
in~$L^2_{\star}(0,T)$ are only defined up to null sets, we let~$u$ belong
to~$U(A,b)$ if any only if there exists any representative of~$u$ with
switching points~$t_1,\dots,t_\sigma$ satisfying~$At\le b$.

For simplicity, we assume that the constraints~$At\le b$ imply~$0\le
t_1\le \dots\le t_\sigma\le T$, so that~$u$ switches up at~$t_i$ for
odd~$i$ and down for even~$i$.  However, we do not require that any of
the inequalities~$0\le t_1\le \dots\le t_\sigma\le T$ be strict. It is
thus allowed to switch up immediately at zero, to switch multiple
times at the same time point (so that the switchings neutralize each
other), or to leave some of the switchings to time point~$T$ (so that
they become irrelevant). In particular, while the switching points
uniquely determine~$u\in BV_{\star}(0,T;\{0,1\})$, the vector of switching
points belonging to some given~$u\in BV_{\star}(0,T;\{0,1\})$ may not be
unique.

\subsection{Linearization}

Different from the cases considered in the previous sections, there is
no obvious linear formulation of~$U(A,b)$ in the original space, due
to the non-linear connection between the values of the function~$u$
and the switching points of~$u$. The following model
describing~$U(A,b)$ by the use of additional controls can thus be seen
as a linearization of~$U(A,b)$:
\begin{subequations}\label{eq:switch}
\begin{align}
  U(A,b)^{\dlin}:=\Big\{ & u\in L^2_{\star}(0,T;\{0,1\}),~z^{(1)},\dots,z^{(\sigma)}\in L^2_{\star}(0,T;[0,1])\colon\notag\\
  & Dz^{(i)}\ge 0 \,\quad\quad\text{ for all }i=1,\dots,\sigma,\label{eq:switch:mono}\\
  & z^{(i)}\le z^{(i-1)} \quad\text{ for all }i=2,\dots,\sigma,\label{eq:switch:order}\\
  & u=\sum_{i=1}^\sigma (-1)^{i+1}z^{(i)},\label{eq:switch:connect}\\
  & \sum_{i=1}^\sigma a_{ji}\int_0^T (1-z^{(i)}(t))\,\d t\le b_j\text{ for all }j=1,\dots,m\Big\}\;.\label{eq:switch:constraints}
\end{align}
\end{subequations}
Note that the constraints~\eqref{eq:switch:constraints} fit into the
framework of Condition~(e2) of Definition~\ref{def:extform} by
considering the left and right hand side as constant
functions. However, this formulation contains binarity constraints
on~$u$. Still, we can show
\begin{theorem}\label{thm:lin}
  The projection of~$U(A,b)^{\dlin}$ to the first
  coordinate agrees with~$U(A,b)$.
\end{theorem}
\begin{proof}
  Given~$u\in U(A,b)$, let~$t_1,\dots,t_\sigma$ be a vector of
  switching points of~$u$ satisfying~$At\le
  b$. Define~$z^{(i)}:=\chi_{(t_i,T)}$ for
  all~$i=1,\dots,\sigma$. Then~$z^{(1)},\dots,z^{(\sigma)}\in
  BV_{\star}(0,T;[0,1])$ and~\eqref{eq:switch:mono}, \eqref{eq:switch:order}, as
  well as~\eqref{eq:switch:connect} are obviously satisfied. Moreover, we
  obtain
  \[
  \textstyle\int_0^T(1-z^{(i)}(t))\,\d t=t_i
  \]
  for all $i=1,\dots,\sigma$ by construction, so
  that~\eqref{eq:switch:constraints} reduces to~$At\le b$. In summary,
  we thus have~$(u,z^{(1)},\dots,z^{(\sigma)})\in U(A,b)^{\dlin}$.
  
  For showing the other direction, assume
  that~$(u,z^{(1)},\dots,z^{(\sigma)})\in
  U(A,b)^{\dlin}$. First note that
  constraint~\eqref{eq:switch:mono}, together with~$z^{(i)}\in [0,1]$
  \aei~$(-\infty,T)$, implies~$|z^{(i)}|_{BV(-\infty,T)}\le 1$, so
  that~$|u|_{BV(-\infty,T)}\le \sigma$ by~\eqref{eq:switch:connect}. From
  the binarity of~$u$, it follows that~$u$ has finitely many switching
  points~$0\le t_1<\dots< t_r< T$ with~$r\le\sigma$. Define~$t_i:=T$
  for~$i=r+1,\dots,\sigma$. It remains to show that~$At\le b$.

  For this, we first show by induction that~$z^{(i)}=0$
  \aei~$(-\infty,t_i)$ for~$i=\sigma,\dots,1$. If~$i$ is even, we
  have~$z^{(j)}=0$ \aei~$(t_{i-1},t_i)$
  for~$j=i+1,\dots,\sigma$ by the induction hypothesis.  So
  from~\eqref{eq:switch:connect} and~\eqref{eq:switch:order} we obtain
  \[
  1=u=\underbrace{z^{(1)}}_{\le 1}\underbrace{-z^{(2)}+z^{(3)}}_{\le
    0}-\dots-\underbrace{z^{(i-2)}+z^{(i-1)}}_{\le 0}-z^{(i)}\le
  1-z^{(i)}\text{ \aei\ }(t_{i-1},t_i)
  \]
  and hence~$z^{(i)}=0$
  \aei~$(t_{i-1},t_i)$. Using~\eqref{eq:switch:mono}, this
  implies~$z^{(i)}=0$ \aei~$(-\infty,t_i)$.  Now let~$i$ be odd. Then
  \[
  0=u=\underbrace{z^{(1)}-z^{(2)}}_{\ge
    0}+\dots+\underbrace{z^{(i-2)}-z^{(i-1)}}_{\ge 0}+z^{(i)}\ge
  z^{(i)}\text{ \aei\ }(t_{i-1},t_i)
  \]
  and hence~$z^{(i)}=0$ \aei~$(t_{i-1},t_i)$, which again
  implies~$z^{(i)}=0$ \aei~$(-\infty,t_i)$ by~\eqref{eq:switch:mono}.

  We next show~$z^{(i)}=\chi_{(t_i,T)}$ inductively
  for~$i=1,\dots,\sigma$. First, let~$i$ be
  odd. By~\eqref{eq:switch:connect}, we have $1=u=z^{(i)}$
  \aei~$(t_{i},t_{i+1})$, since by the induction
  hypothesis~$z^{(j)}=1$ for~$j<i$ and we have already
  shown~$z^{(j)}=0$ for~$j>i$. Thus~$z^{(i)}=\chi_{(t_i,T)}$.
  Similarly, for even~$i$ we obtain~$0=u=1-z^{(i)}$
  on~$(t_{i},t_{i+1})$ and hence again~$z^{(i)}=1$
  \aei~$(t_{i-1},t_i)$, showing again that~$z^{(i)}=\chi_{(t_i,T)}$.
  
  In summary, we have~$\int_0^T(1-z^{(i)}(t))\,\d t=t_i$ for
  all~$i=1,\dots,\sigma$, so that~\eqref{eq:switch:constraints}
  implies~$At\le b$ and hence~$u\in U(A,b)$.
\end{proof}
A closer look at this proof reveals that the main difficulty was to
derive the binarity of the auxiliary
controls~$z^{(1)},\dots,z^{(\sigma)}$,
yielding~$z^{(i)}=\chi_{(t_i,T)}$. If binarity
of~$z^{(1)},\dots,z^{(\sigma)}$ is required directly in the model,
then the above result follows much more easily.

In terms of Definition~\ref{def:extform}, the
set~$U(A,b)^{\textnormal{lin,rel}}$, resulting from
replacing~$\{0,1\}$ by~$[0,1]$ in~$U(A,b)^{\dlin}$, satisfies
Condition~(e2). Moreover, the resulting model is small, assuming
that~$A$ and~$b$ are part of the problem input, and compatible with
discretization, since each~$u\in U(A,b)_N$ can have switching points
only in~$\{0,\tfrac TN,\dots,(N-1)\tfrac TN\}$, so that the
construction in the proof of Theorem~\ref{thm:lin}
yields~$z^{(1)},\dots,z^{(\sigma)}\in L^2_{\star}(0,T)_N$.  Hence, an
obvious question is whether~$U(A,b)^{\textnormal{lin,rel}}$ is an
extended formulation for~$\overline\conv(U(A,b))$, i.e., also
satisfies Condition~(e1).  The answer to this question is negative. In
fact, even for the special case of minimum dwell-time constraints
considered in the previous section, the model~\eqref{eq:switch} does
not yield a complete description.

\begin{example}
To see this, consider the minimum dwell-time instance defined
by~$T=4$, $L=2$, and~$l=0$, so that~\eqref{eq:switch:constraints}
essentially reduces to~$\int_0^T(z^{(1)}(t)-z^{(2)}(t))\,\d t\ge
2$. Let~$u=z^{(1)}-z^{(2)}$ with~$z^{(1)}=\chi_{(1,T)}$
and~$z^{(2)}=\tfrac 12\chi_{(2,T)}$. Then~$(u,z^{(1)},z^{(2)})$
satisfies all constraints in~\eqref{eq:switch} except for the binarity
of~$u$, but~$u\not\in\overline\conv(\DM)$, since all elements of the
latter set satisfy~$u(2\tfrac 14)\ge u(1\tfrac 14)-u(\tfrac 34)$.\qed
\end{example}

However, it is possible to model the constraints~$z-V_L(z)\le u$
and~$z-V_{l}(z)\le {\bf 1}-V_{l}(u)$ of the extended
formulation~$\DM^{\dlin}$ in the variable space of~\eqref{eq:switch}:
using~$z=\sum_{\indodd{i}}z^{(i)}$ and
hence~$z-u=\sum_{\indeven{i}}z^{(i)}$, we obtain
\[
\sum_{\indeven{i}}z^{(i)}\le\sum_{\indodd{i}}V_L(z^{(i)}),\quad
\sum_{\indodd{i}}z^{(i)}\le \mathbf{1}+\sum_{\indeven{i}}
V_{l}(z^{(i)})
\]
as another extended formulation for the case of dwell-time
constraints.

In the following section, we try to convince the reader that a small
extended formulation which is compatible with discretization most
likely does not exist for general linear switching point constraints, even in
the weak sense of Definition~\ref{def:weakextform}.

\subsection{Negative Results}

We now show that a small and compatible extended formulation cannot
exists for general linear switching point constraints unless
P$\,=\,$NP. Using Theorem~\ref{thm:extended}, it suffices to show that
it is NP-complete to decide whether~$U(A,b)_N\neq\emptyset$ for
given~$A$ and~$b$.  For all hardness proofs, we use reductions from
the following elementary decision problem:
\begin{center}
  (BPF) ~ Given~$B\in\Q^{m\times n}$ and~$d\in\Q^m$, does there exist some~$x\in\{0,1\}^n$ with $Bx\le d$?
\end{center}
It is well-known that~(BPF) is NP-complete~\cite{karp72}. We first show
\begin{theorem}\label{thm:hard1}
  Assume that~$T\in\Q_+$, $\sigma,m\in\N$, $A\in\Q^{m\times\sigma}$,
  $b\in\Q^m$, and~$N\in\N$ are part of the input. Then it is
  NP-complete to decide whether~$U(A,b)_N\neq\emptyset$.
\end{theorem}
\begin{proof}
  We show the statement by reduction from (BPF). For this, we
  set~$T=n$ and~$\sigma=2n$. For all~$i=1,\dots,n$, we add the
  switching point constraints~$t_{2i-1}=i-1$ and $i-1\le t_{2i}\le
  i$. In words, the control~$u$ switches up at~0, then down again
  between~0 and~1, up again at~1, and so on. So far, all these
  switchings are independent. We will model the variable~$x_i$
  by~$t_{2i}-(i-1)\in[0,1]$. Substituting all~$x_i$ in~$Bx\le d$ by
  these expressions, we obtain another set of linear constraints
  in~$t$, which we add to the switching point constraints. This
  concludes the construction of~$A$ and~$b$, which can obviously be
  done in polynomial time.

  Now let~$N=n$. Then all switching points of~$u$ belong
  to~$\{0,\dots,n\}$, hence~$t_{2i}-(i-1)\in\{0,1\}$
  and~$t_{2i}-(i-1)=u_{i}$. Thus, by construction, the given instance
  of~(BPF) has a feasible solution if and only if~$U(A,b)_N\neq
  \emptyset$. Clearly, the problem of deciding whether~$U(A,b)_N$ is
  non-empty belongs to NP, the certificate being an element
  of~$U(A,b)_N$.
\end{proof}
Now using Theorem~\ref{thm:hard1} and
Theorem~\ref{thm:extended}, considering, e.g, the objective
function~$c=\mathbf{0}$, we immediately obtain
\begin{corollary}
  A small extended formulation of~$U(A,b)$ that is compatible with
  discretization does not exist unless~P$\,=\,$NP.
\end{corollary}
The proof of Theorem~\ref{thm:hard1} relies on the possibility to
choose the grid size~$N$ in the reduction, i.e., on the fact that~$N$
is part of the input. When keeping the same instance but considering
finer grids, it is no longer true that~$u$ must be constant
between~$i-1$ and~$i$. Indeed, assuming that~$N$ is a multiple of~$T$,
the function~$u$ can switch at any point~$(i-1)+j\tfrac TN$
for~$j=0,\dots,\tfrac NT$. Thus~$t_{2i}-(i-1)$ is no longer binary,
but can take any value in~$\{j\tfrac TN\mid j=0,\dots,\tfrac
NT\}$. Hence, for large enough~$N$, any vertex of the polytope~$\{x\in
[0,1]^n\colon Bx\le d\}$ can be represented by~$x_i=t_{2i}-(i-1)$
since~$A$ and~$b$ are rational, so that the decision
whether~$U(A,b)_N\neq \emptyset$ reduces to deciding feasibility of a
linear program, which can be done in polynomial time. In other words,
Theorem~\ref{thm:hard1} does not rule out the existence of a small
extended formulation for~$U(A,b)$ that is only \emph{weakly}
compatible with discretization.

\begin{example}\label{ex:vc}
  It can happen that different discretizations of the same problem are
  tractable or NP-hard depending on the choice of~$N$, and the two
  situations may even alternate. As an example, consider the following
  fractional version of the vertex cover problem: given a simple
  graph~$G=(V,E)$, $\gamma\in\N$ and~$K\in\Q$, decide whether there exists a
  solution~$x\in\R^V$ of
  \begin{equation*}
    \begin{array}{rcll}
      \sum_{v\in V}x_v & \le & K\\
      x_v+x_w & \ge & 1 & \forall (v,w)\in E\\
      x_v & \in & [0,1] & \forall v\in V
    \end{array}
  \end{equation*}
  such that all entries of~$x$ are integer multiples of~$\nicefrac
  1\gamma$. Theorem~\ref{thm:hard1} shows that this problem can be
  reduced to deciding whether~$U(A,b)_N=\emptyset$ with~$N=\gamma n$,
  for appropriate~$A$ and~$b$. Since the vertex cover polytope is
  half-integral, meaning that all vertices have entries being
  multiples of~$\nicefrac 12$, the above problem reduces to a linear
  program for even~$\gamma$. For the problem constructed in the proof
  of Theorem~\ref{thm:hard1}, it can thus be decided in polynomial
  time whether~$U(A,b)_N\neq\emptyset$ whenever~$N$ is an even
  multiple of~$n$. However, the same decision problem turns
  NP-complete when~$N$ is an odd multiple of~$n$. For this, it
  suffices to show that the above fractional vertex cover problem is
  NP-complete for all odd values of~$\gamma$, which is done in
  Appendix~\ref{app:A}.
  \qed
\end{example}
However, by adding an objective function to the problem, we can show
that even a weakly compatible extended formulation for~$U(A,b)$ cannot
exist in general.
\begin{theorem}\label{thm:noweak}
  A small extended formulation for~$U(A,b)$ that is weakly
  compatible with discretization cannot exists unless~P$\,=\,$NP.
\end{theorem}
\begin{proof}
  Consider the problem
  \begin{equation}\label{eq:e}
    \left\{\quad\begin{aligned}
    \min~~~ & \int_0^T c(t) u(t)\;\d t\\
    \text{s.t.}~~~ & u\in U(A,b)
    \end{aligned}\right.
  \end{equation}
  for the function~$c\in L^2(0,T)$ defined by~$c(t)=\tfrac 12-(t-\lfloor
  t\rfloor)$. Then~$c_N$ can be computed efficiently for each~$N$ and
  the discretized problem reads
  \begin{equation}\label{eq:en}
    \left\{\quad\begin{aligned}
    \min~~~ & \tfrac 1N\sum_{i=1}^N (c_N)_i u_i\\
    \text{s.t.}~~~ & u\in\R^N,~\overline u\in U_N(A,b)\;.
    \end{aligned}\right.
  \end{equation}
  It suffices to show that (BPF) can be polynomially reduced to
  Problem~\eqref{eq:en} for~$N=\ell(n)n$ whenever~$\ell(n)$ is
  polynomial in~$n$. Indeed, under the assumption that a small and
  weakly compatible extended formulation exists, this yields an
  efficient algorithm for deciding~(BPF) as follows: First,
  choose~$M=n$ and efficiently compute some~$\ell$ as in
  Definition~\ref{def:compat}. Since~$\ell$ is required to be
  polynomial in~$M$, we would then have that~(BPF) can be polynomially
  reduced to Problem~\eqref{eq:en} for~$N=\ell M$, and by the
  compatibility assumption, Problem~\eqref{eq:en} can be solved in
  polynomial time for this~$N$. This implies that (BPF) can be solved
  in polynomial time and hence~P$\,=\,$NP.

  In order to construct the desired polynomial reduction, let an
  instance of~(BPF) be given by~$B\in\Q^{m\times n}$ and~$d\in\Q^m$.
  We define~$T$, $\sigma$, and the switching point constraints exactly
  as in the first part of the proof of Theorem~\ref{thm:hard1}. We now
  claim that the given instance of~(BPF) is feasible if and only if
  the constructed instance of Problem~\eqref{eq:en} for~$N=\ell(n)n$
  has an optimal value of zero. For this, the objective value
  of~$u\in\R^N$ with~$\overline u\in U_N(A,b)$ can be computed as
  follows: for~$i=1,\dots,n$, we have
  \begin{eqnarray*}
    \sum_{j=1}^{\ell(n)} (c_{N})_{(i-1)\ell(n)+j}u_{(i-1)\ell(n)+j}
    & = & \sum_{j=1}^{(t_{2i}-(i-1))\ell(n)}(c_{N})_{(i-1)\ell(n)+j}\\
    & = & \int_{i-1}^{t_{2i}} c(t)\,\d t\\
    & = & \int_{0}^{t_{2i}-(i-1)} (\tfrac 12-t)\,\d t\\
    & = & \tfrac 12 \big(t_{2i}-(i-1)\big)\big(1-(t_{2i}-(i-1))\big)\;,
  \end{eqnarray*}
  so that
  \begin{eqnarray*}
    c_N^\top u & = & \sum_{i=1}^n\sum_{j=1}^{\ell(n)} (c_{N})_{(i-1)\ell(n)+j}u_{(i-1)\ell(n)+j}\\
    & = & \tfrac 12\sum_{i=1}^n \big(t_{2i}-(i-1)\big)\big(1-(t_{2i}-(i-1))\big)\;.
  \end{eqnarray*}
  The latter expression is always non-negative,
  since~$t_{2i}-(i-1)\in[0,1]$ for all~$i=1,\dots,n$. It follows that
  all~$u\in U(A,b)_N$ have a non-negative objective value in the
  constructed instance, and the objective value is zero if and only
  if~$t_{2i}-(i-1)\in\{0,1\}$ for all~$i=1,\dots,n$. This concludes
  the proof.
\end{proof}
A closer look at the proof of Theorem~\ref{thm:noweak} reveals that the
difficulty of linear optimization over~$U(A,b)_N$ does \emph{not} stem
from the binarity of the switch~$u$, but from the non-convex relation
between the switching points of~$u$ and the value of~$u$ at a given
point~$t\in[0,T]$. This however does not rule out the existence of
small and compatible extended formulations in special cases, as shown by
Theorem~\ref{thm:eextended}.

\section*{Conflicts of Interest}

The author declares that he has no conflicts of interest.

\bibliographystyle{plain}
\bibliography{reference}

\appendix

\section{Fractional vertex cover}\label{app:A}

We claim that for odd~$\gamma\in\N$, the following decision problem is
NP-complete: given a simple graph~$G=(V,E)$ and~$K\in\Q$, decide
whether there exists a solution~$x\in\Q^V$ of
\begin{equation}\tag{$\gamma$-VC}\label{eq:fvc}
  \left\{\quad
  \begin{array}{rcll}
    \sum_{v\in V}x_v & \le & K\\[1ex]
    x_v+x_w & \ge & 1 & \forall (v,w)\in E\\[1ex]
    x_v & \in & [0,1] & \forall v\in V\\[1ex]
    x_v & \in & \tfrac 1\gamma \Z & \forall v\in V\;.
  \end{array}
  \right.
\end{equation}
To show this claim, we reduce the NP-complete decision variant of the
(ordinary) vertex cover problem to the above problem. So given an
instance of the vertex cover problem, i.e., a graph~$G'=(V',E')$
and~$K'\in\N$, we construct~$G=(V,E)$ from~$G'$ as follows: for
each~$v\in V'$, we add three new vertices~$v^{(1)},v^{(2)},v^{(3)}$
and four new edges
\[
(v,v^{(1)}),(v^{(1)},v^{(2)}),(v^{(2)},v^{(3)}),(v^{(3)},v^{(1)})\;.
\]
In words, we add a triangle~$T_v$ for each vertex~$v\in V'$ and
connect it to~$v$ by a bridge. We now claim that~$G'$ has a vertex
cover of size at most~$K'$ if and only if~\eqref{eq:fvc} has a
solution for~$G$ and~$K:=\tfrac 1\gamma K'+2|V'|$.

First assume that~$S\subseteq V'$ is a vertex cover of~$G'$
with~$|S|\le K'$. Then the following vector is a solution
of~\eqref{eq:fvc}: set~$x_v=\tfrac{\gamma+1}{2\gamma}$ for~$v\in
S\cup\{v^{(1)},v^{(2)}\mid v\in V'\}$
and~$x_v=\tfrac{\gamma-1}{2\gamma}$ otherwise. Indeed, it is easy to
verify that~$x$ satisfies the covering constraints. For the
cardinality constraint, we have
\[
\sum_{v\in V}
x_v=\tfrac{\gamma+1}{2\gamma}(|S|+2|V'|)+\tfrac{\gamma-1}{2\gamma}(|V'|-|S|+|V'|)=\tfrac
1\gamma |S|+2|V'|\le K\;.
\]
For the other direction, let~$x\in\Q^V$ solve~\eqref{eq:fvc}. We may
assume~$x_{v^{(1)}}=\tfrac{\gamma+1}{2\gamma}$ for all~$v\in
V'$. Indeed, using a smaller value does not allow to decrease the
costs of covering~$T_v$ while contributing less to
cover~$(v,v^{(1)})$. On the other hand, using a larger value increases
the costs of covering~$T_v$ by at least as much as it would cost to
increase~$x_v$ instead. So we can assume~$x_v\ge
\tfrac{\gamma-1}{2\gamma}$ for all~$v\in V$ now. In particular, it
suffices to choose~$x_v\le \tfrac{\gamma+1}{2\gamma}$ to cover all
edges in~$E$, hence
\[
x_v\in\{\tfrac{\gamma-1}{2\gamma},\tfrac{\gamma+1}{2\gamma}\}\quad\forall
v\in V'
\]
without loss of generality. Then it is easy to verify that the set
$S:=\{v\in V\mid x_v=\tfrac{\gamma+1}{2\gamma}\}$ is a vertex cover
of~$G'$ of size at most~$K'$.\qed

\end{document}